\documentclass[oneside,english]{amsart}
\usepackage[T1]{fontenc}
\usepackage[latin9]{inputenc}
\usepackage{amsmath}
\usepackage{amsthm}
\usepackage{amstext}
\usepackage{amssymb}
\usepackage{esint}
\usepackage{graphicx}
\usepackage{enumerate}
\usepackage{amscd}

\makeatletter
\newtheorem{theorem}{Theorem}[section]
\newtheorem{lemma}[theorem]{Lemma}

\numberwithin{figure}{section}
\theoremstyle{definition}

\newtheorem{example}[theorem]{Example}

\theoremstyle{remark}
\newtheorem{remark}[theorem]{Remark}

\numberwithin{equation}{section}
\makeatother

\usepackage{color}

\usepackage{babel}

	\DeclareMathOperator{\loc}{loc}
	
	\DeclareMathOperator*{\esssup}{ess\,sup}

\begin{document}

\title[Dirichlet eigenvalues problem]{Conformal composition operators with applications to Dirichlet eigenvalues}
\author{C. Deneche, V. Pchelintsev}

\address{C. Deneche \newline Tomsk State University
36, Lenin Ave., Tomsk 634050, Russia}
\email{carafdenes@gmail.com}

\address{V. Pchelintsev \newline Tomsk State University
36, Lenin Ave., Tomsk 634050, Russia}
\email{va-pchelintsev@yandex.ru}

\begin{abstract}
This paper is concerned with spectral estimates for the first Dirichlet eigenvalue of the degenerate $p$-Laplace operator in bounded simply connected domains $\Omega \subset \mathbb C$. The proposed approach relies on the conformal analysis of the elliptic operators, which allows us to obtain spectral estimates in domains with non-rectifiable boundaries.
\end{abstract}
\maketitle

\section{Introduction}

This paper is devoted to study the Dirichlet eigenvalue problem for the degenerate $p$-Laplace operator with $p>2$:
\begin{equation}\label{DEP}
-div(|\nabla f|^{p-2}\nabla f) = \lambda_p |f|^{p-2}f\,\, \text{in }\,\, \Omega,  \quad f=0\,\, \text{on}\,\, \partial \Omega,
\end{equation}
in bounded simply connected domains $\Omega \subset \mathbb{C}$ in terms of the conformal geometry of domains $\Omega$.

We define a weak solution to this spectral problem as follows: a function $f\in W_0^{1,p}(\Omega)$ is called a weak solution to the Dirichlet eigenvalue problem for the degenerate $p$-Laplace operator if there exists a number $\lambda_p$ that such
\begin{equation*}
\iint\limits _{\Omega} (|\nabla f(z)|^{p-2}\nabla f(z) \cdot \nabla g(z))\,dxdy \\
=\lambda_p \iint\limits _{\Omega} |f(z)|^{p-2} f(z) g(z)\,dxdy,
\end{equation*}
for all $g \in W_0^{1,p}(\Omega)$. We refer to $\lambda_p$ as an eigenvalue and $f$ as the eigenfunction corresponding to $\lambda_p$.

It is well known \cite{LP90} that the first Dirichlet eigenvalue $\lambda_{p}^{(1)}(\Omega)$ of the $p$-Laplace operator is simple and isolated for any bounded domain $\Omega \subset \mathbb C$.
By the min-max principle (see, for example, \cite{L,M}), the first
eigenvalue $\lambda_{p}^{(1)}(\Omega)$ can be characterized as
\[
\lambda_{p}^{(1)}(\Omega)=\inf_{f \in W_0^{1,p}(\Omega) \setminus \{0\}}
\frac{\iint\limits _{\Omega} |\nabla f(z)|^p~dxdy}{\iint\limits _{\Omega} |f(z)|^p~dxdy}\,.
\]

It follows that $\lambda_{p}^{(1)}(\Omega)^{-\tfrac{1}{p}}$ coincides with the best constant $A_{p,p}(\Omega)$ in the
$(p,p)$-Sobolev--Poincar\'e inequality
$$
\left(\iint\limits_\Omega |f(z)|^p dxdy\right)^{\frac{1}{^p}} \leq A_{p,p}(\Omega)
\left(\iint\limits_\Omega |\nabla f(z)|^p dxdy\right)^{\frac{1}{p}},\quad \forall f\in W_{0}^{1,p}(\Omega).
$$

Note that, for $p \neq 2$ the first eigenvalue $\lambda_{p}^{(1)}(\Omega)$ is not explicitly known even for simple symmetric domains such as a square or a disc. Thus, it is important to obtain sharp estimates for
$\lambda_{p}^{(1)}(\Omega)$, which depends on the intrinsic geometry of domains. The most famous example of
such an estimate is the so-called Rayleigh--Faber--Krahn inequality (see, for example, \cite{B99,IL}) asserting that the disc
$\Omega^{\ast}$ minimizes the first Dirichlet eigenvalue of the $p$-Laplace operator among all planar domains $\Omega$ of the
same area:
\[
\lambda_{p}^{(1)}(\Omega) \geq \lambda_{p}^{(1)}(\Omega^*).
\]

In bounded domains, other lower bounds for the first Dirichlet eigenvalue
of the $p$-Laplace operator in terms of the Cheeger constant and inradius were obtained in \cite{AA,KF03,GP}.

In the present paper, using the geometric theory of composition operators \cite{GU09,U93,VU02}, we study the spectral
estimates in bounded simply connected domains $\Omega \subset \mathbb C$ under suitable geometric assumptions.

The first main result of the paper is formulated as follows:
\begin{theorem}\label{th-1}
Let $\Omega$ be a conformal $\alpha$-regular domain about the unit disc $\mathbb D$.
Then for any $p>2$ the following inequality holds
\begin{equation*}
\frac{1}{\lambda_p^{(1)}(\Omega)} \leq \\
 \inf_{q \in (q^*,2)} \left\{ A_{\frac{p \alpha}{\alpha -1}, q}^p(\mathbb D) \cdot \pi^{\frac{p(2 - q)}{2q}}\right\} \cdot |\Omega|^{\frac{p - 2}{2}} \cdot ||J_{\varphi}\,|\,L^{\alpha}(\mathbb D)||,
\end{equation*}
where $q^*=2\alpha p/(\alpha p+ 2(\alpha-1))$, $\alpha >1$ and
\[
A_{\frac{p \alpha}{\alpha -1}, q}(\mathbb D) \leq \left( \dfrac{q - 1}{2 - q} \right)^{\!\frac{q - 1}{q}}
\dfrac{(\sqrt{\pi}\cdot \sqrt[q]{2})^{-1}}
{\sqrt{\Gamma(2/q)\,\Gamma(3 - 2/q)}}.
\]
\end{theorem}

If the domains under consideration are the images of the unit disc under $K$-quasiconformal
mappings $\varphi:\mathbb C \to \mathbb C$, then Theorem~\ref{th-1} can be formulated as follows:
\begin{theorem}
Let $\Omega\subset\mathbb C$ be a $K$-quasidisc. Then
$$
\lambda_p^{(1)}(\Omega)\geq \frac{M_p(K)}{|\Omega|^{\frac{p}{2}}}=\frac{M^{\ast}_p(K)}{R_{\ast}^{p}},
$$
where $R_{\ast}$ is a radius of a disc $\Omega^{\ast}$ of the same area as $\Omega$ and $M^{\ast}_p(K)=M_p(K)\pi^{-{p/2}}$.
The constant $M_p(K)$ is defined by~\eqref{M(K)} and depends only on $p$ and a quasiconformality coefficient $K$ of $\Omega$.
\end{theorem}

Let $\Omega,\widetilde{\Omega} \subset \mathbb C$ be bounded simply connected domains. Then a domain $\Omega$ is called a conformal $\alpha$-regular domain about a domain $\widetilde{\Omega}$, $\alpha>1$, if
\[
\iint\limits_{\widetilde{\Omega}} J(w,\varphi)^{\alpha}dudv<\infty,
\]
where $J(w,\varphi)$ is a Jacobian of conformal mapping $\varphi :\widetilde{\Omega} \to \Omega$ \cite{BGU15}.

This concept does not depend on choice of a conformal mapping $\varphi:\widetilde{\Omega}\to\Omega$ and can be reformulated in terms of the conformal radii. Namely
\[
\iint\limits_{\widetilde{\Omega}} J(w,\varphi)^{\alpha}dudv=
\iint\limits_{\widetilde{\Omega}} \left( \frac{R_{\Omega}(\varphi(w))}{R_{\widetilde{\Omega}}(w)}\right)^{2\alpha}dudv,
\]
where $R_{\widetilde{\Omega}}(w)$ and $R_{\Omega}(\varphi(w))$ are conformal radii of $\widetilde{\Omega}$ and $\Omega$ \cite{A20}.
The domain $\Omega$ is a conformal regular domain if it is an $\alpha$-regular domain for some $\alpha >1$. This class of domains includes, in particular, domains with a Lipschitz boundary and fractal domains of the snowflake type \cite{GPU17_2}. The Hausdorff dimension of the conformal regular domain's boundary can be any number in [1, 2) \cite{HKN}.

\begin{remark}
Theorem~\ref{th-1} can be reformulated in terms of the conformal radii
$R_{\Omega}(\cdot)$ of a domain $\Omega$.
\end{remark}

The applications of composition operators on Sobolev spaces were given in \cite{BGU15,DP25,GPU18,GU16,GU161}, where the Dirichlet and Neumann eigenvalue problems for simply connected conformal regular domains were considered.

The paper is organized as follows. Section 2 contains basic facts about composition operators
on Sobolev spaces and the theory of conformal mappings. In Section 3, we prove the
Sobolev--Poincar\'e inequality for conformal regular domains. In Section 4, we apply the results of Section
3 to lower estimates of the first Dirichlet eigenvalue of the degenerate $p$-Laplace operator in conformal regular
domains. In Section 5, we refine these estimates for quasidiscs.

\section{Sobolev spaces and conformal mappings}

\subsection{Sobolev spaces}
Let $E \subset \mathbb{C}$ be a measurable set and $h:E \to \mathbb R$ be a positive a.e. locally integrable function i.e. a weight. The weighted Lebesgue space $L^p(E,h)$, $1\leq p<\infty$,
is the space of all locally integrable functions with the finite norm
$$
\|f\,|\,L^{p}(E,h)\|= \left(\iint\limits_E|f(z)|^ph(z)\,dxdy\right)^{\frac{1}{p}}< \infty.
$$
If $h=1$, we write $L^p(E,h)=L^{p}(E)$.

Let $\Omega$ be an open subset of $\mathbb{C}$. The Sobolev space $W^{1,p}(\Omega)$, $1\leq p<\infty$ is defined
as a Banach space of locally integrable weakly differentiable functions
$f:\Omega\to\mathbb{R}$ equipped with the following norm:
\[
\|f\,|\, W^{1,p}(\Omega)\|=\|f\,|\,L^{p}(\Omega)\|+\|\nabla f\,|\, L^{p}(\Omega)\|,
\]
where $\nabla f$ is the weak gradient of the function $f$. Recall that the Sobolev space $W^{1,p}(\Omega)$ coincides with the closure of the space of smooth functions $C^{\infty}(\Omega)$ in the norm of $W^{1,p}(\Omega)$.

The Sobolev space $W^{1,p}_{\loc}(\Omega)$ is defined as follows: $f\in W^{1,p}_{\loc}(\Omega)$
if and only if $f\in W^{1,p}(U)$ for every open and bounded set $U\subset  \Omega$ such that
$\overline{U}  \subset \Omega$, where $\overline{U} $ is the closure of the set $U$.

The homogeneous seminormed Sobolev space $L^{1,p}(\Omega)$, $1\leq p<\infty$,
is the space of all locally integrable weakly differentiable functions equipped
with the following seminorm:
\[
\|f\mid L^{1,p}(\Omega)\|=\|\nabla f\,|\, L^{p}(\Omega)\|.
\]

The Sobolev space $W^{1,p}_{0}(\Omega)$, $1 \leq p< \infty$, is the closure in the $W^{1,p}(\Omega)$-norm of the
space $C^{\infty}_{0}(\Omega)$ of all infinitely continuously differentiable functions with compact support in $\Omega$.

We consider the Sobolev spaces as Banach spaces of equivalence classes of functions up to a set of $p$-capacity zero \cite{M,U93}.

 Recall that a mapping $\varphi:\widetilde{\Omega}\to\Omega$ belongs to $W^{1,p}_{\loc}(\widetilde{\Omega})$,
$1\leq p\leq\infty$, if its coordinate functions $\varphi_j$ belong to $W^{1,p}_{\loc}(\widetilde{\Omega})$, $j=1,\dots,n$.
In this case the formal Jacobi matrix
$D\varphi(x)$ and its determinant (Jacobian) $J(x,\varphi)$ are well defined at
almost all points $x\in \widetilde{\Omega}$. The norm $|D\varphi(x)|$ of the matrix
$D\varphi(x)$ is the norm of the corresponding linear operator.

We say that
a homeomorphism $\varphi:\widetilde{\Omega}\to\Omega$ induces by the composition rule $\varphi^{\ast}(f)=f\circ\varphi$ a bounded composition operator
\[
\varphi^{\ast}:L^{1,p}(\Omega)\to L^{1,q}(\widetilde{\Omega}),\,\,\,1\leq q\leq p\leq\infty,
\]
if the composition $\varphi^{\ast}(f)\in L^{1,q}(\widetilde{\Omega})$
is defined quasi-everywhere (up to a set of $q$-capacity zero) in $\widetilde{\Omega}$ and there exists a constant $K_{p,q}(\widetilde{\Omega})<\infty$ such that
\[
\|\varphi^{\ast}(f)\mid L^{1,q}(\widetilde{\Omega})\|\leq K_{p,q}(\widetilde{\Omega})\|f\mid L^{1,p}(\Omega)\|
\]
for any function $f\in L^{1,p}(\Omega)$ \cite{U93}.

The main result of \cite{U93} gives the analytic description of composition operators on Sobolev spaces. Let us reformulate one in the case of diffeomorphisms.

\begin{theorem}\label{CompTh}
Let $\varphi:\widetilde{\Omega}\to\Omega$ be a  diffeomorphism
between two domains $\widetilde{\Omega}$ and $\Omega$. Then $\varphi$ induces a bounded composition
operator
\[
\varphi^{\ast}:L^{1,p}(\Omega)\to L^{1,q}(\widetilde{\Omega}),\,\,\,1\leq q< p<\infty,
\]
 if and only if
$$
K_{p,q}(\widetilde{\Omega})=\left(\iint\limits_{\widetilde{\Omega}} \left(\frac{|D\varphi(w)|^p}{|J(w,\varphi)|}\right)^\frac{q}{p-q}dudv\right)^\frac{p-q}{pq}<\infty.
$$
\end{theorem}

\subsection{Conformal mappings} Let $\Omega$ and $\widetilde{\Omega}$ be domains in $\mathbb{C}$. A mapping $\varphi:\widetilde{\Omega}\to\Omega$ is called \emph{conformal}~\cite{IM} if
$\varphi \in W^{1,2}_{\mathrm{loc}}(\widetilde{\Omega})$ and
\[
|D\varphi(w)|^2 = J(w,\varphi) \quad \text{a.e. in } \widetilde{\Omega}.
\]

Note that, by the Cauchy--Riemann equations, which hold almost everywhere in $\Omega$, any conformal mapping is holomorphic and therefore smooth: $\varphi \in C^\infty(\widetilde{\Omega})$~\cite{IM}.

The following theorem gives the estimate of the norm of composition operators on Sobolev spaces
induces by conformal mappings in any simply connected domain with finite measure \cite{GPU18}:
\begin{theorem}\label{Th 3.1}
Let $\Omega,\widetilde{\Omega}\subset\mathbb C$ be simply connected domains with finite measure.  Then a conformal mapping $\varphi : \widetilde{\Omega} \to \Omega$ induces a bounded composition operator
$$
\varphi^*: L^{1,p}(\Omega) \to L^{1,q}(\widetilde{\Omega} )
$$
for any $p > 2$ and ${q \in[1,2]}$ with the constant
\begin{equation*}\label{Comp-const}
K_{p,q}(\widetilde{\Omega}) \le |\Omega|^{\frac{p-2}{2p}} |\widetilde{\Omega}|^{\frac{2-q}{2q}}.
\end{equation*}
\end{theorem}

The following result gives a connection between conformal regular domains and composition operators on Sobolev spaces \cite{GPU18}:
\begin{theorem}
\label{th2.2}
Let $\Omega,\widetilde{\Omega} \subset \mathbb{C}$ be simply connected domains.
Then $\Omega$ is a conformal $\alpha$-regular domain about a domain $\widetilde{\Omega}$ if and only if any conformal mapping
$\varphi :\widetilde{\Omega} \to \Omega$ induces a bounded composition operator
\[
\varphi^* : L^{1,p}(\Omega) \to L^{1,q}(\widetilde{\Omega})
\]
for any $p \in (2, +\infty)$ and $q = p\alpha / (p + \alpha - 2)$.
\end{theorem}

\section{Sobolev--Poincar\'e inequality}

Let $\Omega \subset \mathbb{C}$ be a bounded domain. Then $\Omega$ is called an $(r,q)$-Sobolev--Poincar\'e domain, $1 \le r, q\le \infty$,  if there exists a constant $C_{r,q} < \infty$, such that for any function $f \in L^{1,q}(\Omega) $, the $(r,q)$-Sobolev--Poincar\'e inequality
\begin{equation} \label{3.0}
||f\,|\,L^{r}(\Omega)|| \le C_{r,q}||\nabla f\,|\,L^{q}(\Omega)||,
\end{equation}
holds. We denote by $A_{r,q}(\Omega)$ the best constant in this inequality.

The constant $A_{r,q}(\Omega)$ in the inequality \eqref{3.0} can be estimated as:
\begin{equation}\label{SPC}
A_{r,q}(\Omega) =
\begin{cases}
\frac{1}{2\sqrt{\pi}}, & \text{if } q = 1, \\[1.2em]
\displaystyle
\left( \frac{q - 1}{2 - q} \right)^{\!\frac{q - 1}{q}}
\frac{\left(\sqrt{\pi}\cdot \sqrt[q]{2}\right)^{-1}}
{\sqrt{\Gamma(2/q)\,\Gamma(3 - 2/q)}}, & \text{if } 1 < q < 2,\\[1.2em]
\displaystyle
\inf_{l \in \left(\frac{2r}{r+2},\,2\right)}
\left( \frac{l - 1}{2 - l} \right)^{\!\frac{l - 1}{l}}
\frac{\left(\sqrt{\pi}\cdot \sqrt[l]{2}\right)^{-1}|\Omega|^{1/r}}
{\sqrt{l (2/l)\,l(3 - 2/l)}}, & \text{if } q = 2,
\end{cases}
\end{equation}
where $\Gamma(\cdot)$ is the gamma function. The inequality \eqref{3.0} with the constant $A_{r,q}(\Omega)$ was obtained by Federer and Fleming~\cite{FF60} and by Maz'ya~\cite{M60} for $q = 1$,
by Aubin~\cite{A1976} and Talenti~\cite{T76} for $1 < q < 2$,
and for $q = 2$ by Gol'dshtein, Pchelintsev and Ukhlov~\cite{GPU20}.

Next, given Theorem~\ref{Th 3.1} we prove an universal weighted Sobolev--Poincar\'e inequality
which holds in any simply connected planar domain with finite measure.
Denote by $h(z)=J(z,\varphi^{-1})$ the conformal weight defined by a conformal mapping $\varphi:\widetilde{\Omega} \to \Omega$.

\begin{theorem}\label{Th 3.2}
Let $\Omega \subset \mathbb{C}$ be a simply connected domain with finite measure and let $ h(z) = J(z, \varphi^{-1})$ be the conformal weight defined by a conformal mapping $\varphi :\widetilde{\Omega} \to \Omega$.
Suppose that $\widetilde{\Omega}$ be an $(r,q)$-Sobolev-Poincar\'e domain,
then for every function $f \in W_0^{1,p}$, $p>2$, the inequality
\[
\left( \iint\limits_{\Omega} |f(z)|^r h(z)\,dxdy \right)^{\frac{1}{r}}
   \le A_{r,p}(\Omega,h)
   \left( \iint\limits_{\Omega} |\nabla f(z)|^p\,dxdy \right)^{\frac{1}{p}}
\]
holds for any $1 \leq r \leq 2q/(2-q)$ with the constant
\[
A_{r,p}(\Omega,h) \le \inf_{q \in \left(\frac{2r}{r+2},2\right]} \left\{A_{r,q}(\widetilde{\Omega}) |\widetilde{\Omega}|^{\frac{2-q}{2q}}\right\} |\Omega|^{\frac{p-2}{2p}},
\]
where $A_{r,q}(\widetilde{\Omega})$ is as defined in~\eqref{SPC}.

\end{theorem}

\begin{proof}
Let $p > 2$. By the Riemann's mapping theorem there exists a conformal mapping $\varphi:\widetilde{\Omega} \to \Omega$.
So, by Theorem~\ref{Th 3.1}, the composition operator $\varphi^{*}:L^{1,p}(\Omega) \to L^{1,q}(\widetilde{\Omega})$ is bounded and the inequality
\begin{equation}\label{IN2.1}
||f \circ \varphi \,|\, L^{1,q}(\widetilde{\Omega})|| \leq {K}_{p,q}(\widetilde{\Omega}) ||f \,|\, L^{1,p}(\Omega)||
\end{equation}
holds for every function $f \in L^{1,p}(\Omega)$ and for any $q$ such that $1\leq q \leq 2$.

Let $f \in L^{1,p}(\Omega) \cap C^{\infty}_0(\Omega)$. Then, by~\eqref{IN2.1}, the function $f \circ \varphi$ belongs to the Sobolev space $L^{1,q}(\widetilde{\Omega}) \cap C^{\infty}_0(\widetilde{\Omega})$ as the composition of a smooth function with a conformal mapping. Moreover,
\[
f \circ \varphi \in W^{1,q}_0(\widetilde{\Omega}) \cap C^{\infty}_0(\widetilde{\Omega}), \quad q \in [1,2].
\]

Hence, in the domain $\widetilde{\Omega}$, the classical Sobolev-Poincar\'e inequality
\begin{equation}\label{IN2.3}
||f \circ \varphi  \,|\, L^{r}(\widetilde{\Omega})|| \leq A_{r,q}(\widetilde{\Omega}) ||f \circ \varphi \,|\, L^{1,q}(\widetilde{\Omega})||
\end{equation}
holds for any $r$ such that $1 \leq r \leq {2q}/{(2-q)}$.

Next, using the change of variable formula for conformal mappings, we  obtain
\begin{multline}\label{IN2.4}
\left(\iint\limits_\Omega |f(z)|^rh(z)dxdy\right)^{\frac{1}{r}} \\
=
\left(\iint\limits_\Omega |f(z)|^r J(z,\varphi^{-1}) dxdy\right)^{\frac{1}{r}}
= \left(\iint\limits_{\widetilde{\Omega}} |f \circ \varphi(w)|^rdudv\right)^{\frac{1}{r}}.
\end{multline}

So, combining equality~\eqref{IN2.4} and inequality~\eqref{IN2.3} we obtain the inequality
\begin{equation}
\label{IN}
\|f\,|\, L^{r}(\Omega,h)\|
\leq A_{r,q}(\widetilde{\Omega}) ||f \circ \varphi \,|\, L^{1,q}(\widetilde{\Omega})||,
\end{equation}
holds for any $r$ such that $1 \leq r \leq {2q}/{(2-q)}$.

Applying the above Sobolev-Poincar\'e inequality and the composition operator estimate from Theorem~\ref{Th 3.1}, we obtain
\begin{equation}\label{IN3}
\left(\iint\limits_\Omega |f(z)|^rh(z)dxdy\right)^{\frac{1}{r}}
\leq K_{p,q}(\widetilde{\Omega}) A_{r,q}(\widetilde{\Omega}) \left(\iint\limits_\Omega |\nabla f(z)|^p dxdy\right)^{\frac{1}{p}}
\end{equation}
for every function $f \in L^{1,p}(\Omega) \cap C^{\infty}_0(\Omega)$.

Therefore, inequality~\eqref{IN3} holds for any $q \in (2r/(r+2),2]$ if $r \in [1,\,2q/(2-q)]$ and $p\in (2,+ \infty)$.
Thus, we have inequality~\eqref{IN2.1} for smooth functions with constant $A_{r,p}(\Omega,h)$.

Now we extend inequality~\eqref{IN3} for functions from $W^{1,p}_0(\Omega)$.
Let a sequence $\{f_k\}\subset C^{\infty}_0(\Omega)$ be such that it converges to $f$ in the norm of $W^{1,p}_{0}(\Omega)$.
We take $m,\,n \in \mathbb N$ and apply inequality~\eqref{IN3} to the difference $f_m-f_n$ and obtain
\begin{equation}\label{IN2.6}
||f_m-f_n\,|\, L^{r}(\Omega,h)|| \leq A_{r,p}(\Omega,h) ||\nabla f_m- \nabla f_n \,|\, L^{p}(\Omega)||.
\end{equation}
Since the sequence $\{f_k\}$ converges in $W^{1,p}_0(\Omega)$, it is a Cauchy sequence in this space. Therefore, from inequality~\eqref{IN2.6} it follows that this sequence is also a Cauchy sequence in $L^r(\Omega,h)$. In turn, from the completeness of this space the sequence $\{f_k\}$ converges to the same element $f$ in the norm of $L^r(\Omega,h)$.

Thus, we can pass to the limit in inequality~\eqref{IN2.6} as $m \to +\infty$ and obtain the following inequality
\[
||f-f_n\,|\, L^{r}(\Omega,h)|| \leq A_{r,p}(\Omega,h) ||\nabla f- \nabla f_n \,|\, L^{p}(\Omega)||.
\]
From here we have
\begin{multline}\label{IN2.7}
||f\,|\, L^{r}(\Omega,h)|| \leq ||f_n\,|\, L^{r}(\Omega,h)|| + ||f-f_n\,|\, L^{r}(\Omega,h)|| \\
\leq A_{r,p}(\Omega,h) ||\nabla f_n \,|\, L^{p}(\Omega)|| + A_{r,p}(\Omega,h) ||\nabla f- \nabla f_n \,|\, L^{p}(\Omega)|| \\
\leq A_{r,p}(\Omega,h) ||\nabla f \,|\, L^{p}(\Omega)|| + 2A_{r,p}(\Omega,h) ||\nabla f- \nabla f_n \,|\, L^{p}(\Omega)||,
\end{multline}
where for the last inequality we use the following estimate
\[
\left|\, ||\nabla f_n \,|\, L^{p}(\Omega)|| - ||\nabla f \,|\, L^{p}(\Omega)||\,\right|
\leq ||\nabla f - \nabla f_n \,|\, L^{p}(\Omega)||.
\]
Passing to the limit in inequality~\eqref{IN2.7} as $n \to +\infty$ we obtain the following inequality
\[
||f\,|\, L^{r}(\Omega,h)|| \leq A_{r,p}(\Omega,h) ||\nabla f \,|\, L^{p}(\Omega)||,\quad\text{for all}\quad
f \in W^{1,p}_{0}(\Omega),
\]
with the constant
$$
A_{r,p}(\Omega,h)=\inf_{q \in \left(\frac{2r}{r+2},2\right]}\{{K}_{p,q}(\widetilde{\Omega}) \cdot A_{r,q}(\widetilde{\Omega} \}.
$$
\end{proof}

Recall that in conformal regular domains there exists a continuous embedding of weighted Lebesgue spaces $L^r(\Omega, h)$
into the standard Lebesgue spaces $L^s(\Omega)$ for
$s=r(\alpha - 1)/\alpha$ (see~\cite{GU16,GU161}).
\begin{lemma}\label{Lemma 3.3}
Let $\Omega$ be a conformal $\alpha$-regular domain. Then, for every function
$f \in L^r(\Omega, h)$ with $\alpha/(\alpha - 1) \leq r < \infty$, the following inequality holds:
$$
||f\,|\,L^{s}(\Omega)|| \leq \left(\iint\limits_{\widetilde{\Omega}}\ J(w, \varphi)^\alpha~dudv \right)^{{\frac{1}{\alpha}} \cdot \frac{1}{s}} ||f\,|\,L^{r}(\Omega,h)||,
$$
where $s = \frac{\alpha - 1}{\alpha} r$.
\end{lemma}

On the basis of Theorem~\ref{Th 3.2} and Lemma~\ref{Lemma 3.3}, we prove the Sobolev--Poincar\'e inequality with
an explicit upper estimate for the optimal constant in conformal regular domains.
\begin{theorem}\label{th 3.4}
Let $\Omega \subset \mathbb{C}$ be a conformal $\alpha$-regular domain about an $(r,q)$-Sobolev--Poincar\'e domain $\widetilde{\Omega}$. Then for any function $f \in W_0^{1,p}(\Omega)$, $p > 2$, the Sobolev--Poincar\'e inequality
\[
\left( \iint\limits_{\Omega} |f(z) |^s \, dxdy \right)^{\frac{1}{s}}
\leq A_{s,p}(\Omega) \left( \iint\limits_{\Omega} |\nabla f(z)|^p \, dxdy \right)^{\frac{1}{p}}
\]
holds for any $1 \leq s \leq 2q/(2-q) \cdot (\alpha-1)/ \alpha$ with the constant
\begin{multline*}
A_{s,p}(\Omega) \leq \left( \iint\limits_{\widetilde{\Omega}} |J(w, \varphi)|^\alpha \, dudv \right)^{\frac{1}{\alpha} \cdot \frac{1}{s}} A_{r,p}(\Omega, h) \\
\leq \inf\limits_{q \in (q^*,2]} \left\{ A_{\frac{\alpha s}{\alpha - 1}, q}(\widetilde{\Omega}) \cdot |\widetilde{\Omega}|^{\frac{2 - q}{2q}}\right\} \cdot |\Omega|^{\frac{p - 2}{2p}} \cdot \left( \iint\limits_{\widetilde{\Omega}} |J(w, \varphi)|^\alpha \, dudv \right)^{\frac{1}{\alpha} \cdot \frac{1}{s}},
\end{multline*}
where $q^*=2\alpha s/(\alpha s+ 2(\alpha-1))$.
\end{theorem}

\begin{proof}
Let $f \in W_0^{1,p}(\Omega)$, $p > 2$. Then by Theorem \ref{Th 3.2} and Lemma \ref{Lemma 3.3} we get
\begin{multline*}
\left( \iint\limits_{\Omega} |f(z)|^s \, dxdy \right)^{\frac{1}{s}}\\
\leq \left( \iint\limits_{\widetilde{\Omega}} |J(w, \varphi)|^\alpha \, dudv \right)^{\frac{1}{\alpha} \cdot \frac{1}{s}}
\left( \iint\limits_{\Omega} |f(z)|^r h(z) \, dxdy \right)^{\frac{1}{r}}\\
\leq A_{r,p}(\Omega, h) \left( \iint\limits_{\widetilde{\Omega}} |J(w, \varphi)|^\alpha \, dudv \right)^{\frac{1}{\alpha} \cdot \frac{1}{s}}
\left( \iint\limits_{\Omega} |\nabla f(z)|^p \, dxdy \right)^{\frac{1}{p}},
\end{multline*}
for $1 \leq s \leq 2q/(2-q) \cdot (\alpha-1)/ \alpha$.

Since by Lemma \ref{Lemma 3.3} $s = \frac{\alpha - 1}{\alpha} r$ and by Theorem \ref{Th 3.2} $r \geq 1$. Hence $s \geq 1$ and the theorem is proved.
\end{proof}

\section{Lower Estimates for $\lambda_p^{(1)}(\Omega)$}
We consider the Dirichlet eigenvalue problem~\eqref{DEP} in the weak formulation:
\begin{equation*}
\iint\limits _{\Omega} (|\nabla f(z)|^{p-2}\nabla f(z) \cdot \nabla g(z))\,dxdy \\
=\lambda_p \iint\limits _{\Omega} |f(z)|^{p-2} f(z) g(z)\,dxdy,
\end{equation*}
for all $g \in W_0^{1,p}(\Omega)$. By the min-max principle (see, for example, \cite{L,M}), the first
eigenvalue $\lambda_{p}^{(1)}(\Omega)$ can be characterized as
\[
\lambda_{p}^{(1)}(\Omega)=\inf_{f \in W_0^{1,p}(\Omega) \setminus \{0\}}
\frac{\iint\limits _{\Omega} |\nabla f(z)|^p~dxdy}{\iint\limits _{\Omega} |f(z)|^p~dxdy}\,.
\]
In other words, $\lambda_1(\Omega)^{-\frac{1}{p}}$ is the exact constant $A_{p,p}(\Omega)$ in the Sobolev--Poincar\'e inequality
$$
\left(\iint\limits_\Omega |f(z)|^p dxdy\right)^{\frac{1}{^p}} \leq A_{p,p}(\Omega)
\left(\iint\limits_\Omega |\nabla f(z)|^p dxdy\right)^{\frac{1}{p}},\quad \forall f\in W_{0}^{1,p}(\Omega).
$$

\begin{theorem}\label{th 3.5}
Let $\Omega$ be a conformal $\alpha$-regular domain about an $(r,q)$-Sobolev--Poincar\'e domain $\widetilde{\Omega}$,
$r=p \alpha/(\alpha -1)$.
Then for any $p>2$ the following inequality holds
\begin{equation*}
\frac{1}{\lambda_p^{(1)}(\Omega)} \leq \\
 \inf_{q \in (q^*,2]} \left\{ A_{r, q}^p(\widetilde{\Omega}) \cdot |\widetilde{\Omega}|^{\frac{p(2 - q)}{2q}}\right\} \cdot |\Omega|^{\frac{p - 2}{2}} \cdot ||J_{\varphi}\,|\,L^{\alpha}(\widetilde{\Omega})||,
\end{equation*}
where $q^*=2\alpha p/(\alpha p+ 2(\alpha-1))$ and $A_{r, q}(\widetilde{\Omega})$ is as defined in~\eqref{SPC}.
\end{theorem}

\begin{proof}
By the min-max principle and Theorem~\ref{th 3.4} in the case $s=p$, we have
\[
\iint\limits_{\Omega}|f(z)|^p~dxdy \leq A^p_{p,p}(\Omega)
\iint\limits_{\Omega} |\nabla f(z)|^p~dxdy,
\]
where
\[
A_{p,p}(\Omega) \leq \inf\limits_{q \in (q^*,2]} \left\{ A_{\frac{\alpha p}{\alpha - 1}, q}(\widetilde{\Omega}) \cdot |\widetilde{\Omega}|^{\frac{2 - q}{2q}}\right\} \cdot |\Omega|^{\frac{p - 2}{2p}} \cdot ||J_{\varphi}\,|\,L^{\alpha}(\widetilde{\Omega})||^{\frac{1}{p}},
\]
and $q^*=2\alpha p/(\alpha p+ 2(\alpha-1))$.
Thus
\begin{equation*}
\frac{1}{\lambda_p^{(1)}(\Omega)} \leq \\
 \inf_{q \in (q^*,2]} \left\{ A_{r, q}^p(\widetilde{\Omega}) \cdot |\widetilde{\Omega}|^{\frac{p(2 - q)}{2q}}\right\} \cdot |\Omega|^{\frac{p - 2}{2}} \cdot ||J_{\varphi}\,|\,L^{\alpha}(\widetilde{\Omega})||,
\end{equation*}
where $q^*=2\alpha p/(\alpha p+ 2(\alpha-1))$.
\end{proof}

This theorem can be formulated in terms of the conformal radius~\cite{A20}:
\begin{theorem}
Let $\Omega$ be a conformal $\alpha$-regular domain about an $(r,q)$-Sobolev--Poincar\'e domain $\widetilde{\Omega}$,
$r=p \alpha/(\alpha -1)$.
Then for any $p>2$ the following inequality holds
\begin{equation*}
\frac{1}{\lambda_p^{(1)}(\Omega)} \leq \\
 \inf_{q \in (q^*,2]} \left\{ A_{r, q}^p(\widetilde{\Omega}) \cdot |\widetilde{\Omega}|^{\frac{p(2 - q)}{2q}}\right\} \cdot |\Omega|^{\frac{p - 2}{2}}
 \left(\iint\limits_{\widetilde{\Omega}} \left( \frac{R_{\Omega}(\varphi(w))}{R_{\widetilde{\Omega}}(w)}\right)^{2\alpha}dudv\right)^{\frac{1}{\alpha}},
\end{equation*}
where $q^*=2\alpha p/(\alpha p+ 2(\alpha-1))$, $A_{r, q}(\widetilde{\Omega})$ is as defined in~\eqref{SPC},
$R_{\widetilde{\Omega}}(w)$ and $R_{\Omega}(\varphi(w))$ are conformal radii of $\widetilde{\Omega}$ and $\Omega$.
\end{theorem}

\begin{remark}
In the case of the unit disc $\mathbb D \subset \mathbb C$, the conformal radius is defined by the equality $R_{\mathbb D}(w)=1-|w|^2$.
\end{remark}

In the case conformal $\infty$-regular domains, in the similar way we have the following assertion:
\begin{theorem}\label{th 3.6}
Let $\Omega$ be a conformal $\infty$-regular domain about an $(r,q)$-Sobolev--Poincar\'e domain $\widetilde{\Omega}$.
Then for any $p>2$ the following inequality holds
\begin{equation*}
\frac{1}{\lambda_p^{(1)}(\Omega)} \leq   \\
 \inf_{q \in (q^*,2]} \left\{ A_{p, q}^p(\widetilde{\Omega}) \cdot |\widetilde{\Omega}|^{\frac{p(2 - q)}{2q}}\right\} \cdot |\Omega|^{\frac{p - 2}{2}} \cdot ||J_{\varphi}\,|\,L^{\infty}(\widetilde{\Omega})||,
\end{equation*}
where  $q^*=2p/(p+ 2)$ and $A_{p, q}(\widetilde{\Omega})$ is as defined in~\eqref{SPC}.
\end{theorem}

We illustrate this result by the following examples:
\begin{example}
For $n \in \mathbb{N}$, the mapping
\[
\varphi(z)= z + \frac{1}{n} z^n, \quad z = x + iy,
\]
is conformal and maps the unit disc $\mathbb{D}$ onto the domain $\Omega_n$ bounded by an epicycloid of $(n-1)$ cusps, inscribed in the circle $|w| = (n+1)/n$.

Now we estimate the following quantities. A straightforward calculation
yields
\[
\|J_{\varphi}\mid L^{\infty}(\mathbb D)\| = \esssup\limits_{\mathbb D} |1 + z^{n-1}|^2 \leq 4
\]
and
\[
|\Omega_n| \leq \pi \left(\frac{n+1}{n}\right)^2.
\]

Let us remind that for $1<q<2$ the constant $A_{p, q}(\mathbb D)$ can be estimated as
\[
A_{p, q}(\mathbb D)\leq \left( \dfrac{q - 1}{2 - q} \right)^{\!\frac{q - 1}{q}}
\dfrac{\left(\sqrt{\pi}\cdot \sqrt[q]{2}\right)^{-1}}
{\sqrt{\Gamma(2/q)\,\Gamma(3 - 2/q)}}.
\]

Then by Theorem \ref{th 3.6} we have
\begin{multline*}
\frac{1}{\lambda_p^{(1)}(\Omega_n)} \leq
 \inf_{q \in (q^*,2)} \left\{ A_{p, q}^p(\mathbb D) \cdot |\mathbb D|^{\frac{p(2 - q)}{2q}}\right\} \cdot |\Omega_n|^{\frac{p - 2}{2}} \cdot ||J_{\varphi}\,|\,L^{\infty}(\mathbb D)|| \\
\leq  \inf_{q \in (q^*,2)} \left\{ \left[\left( \dfrac{q - 1}{2 - q} \right)^{\!\frac{q - 1}{q}}
\dfrac{\left(\sqrt{\pi}\cdot \sqrt[q]{2}\right)^{-1}}
{\sqrt{\Gamma(2/q)\,\Gamma(3 - 2/q)}}\right]^p \pi^{\frac{p}{q}}\right\} \frac{4}{\pi} \left(\frac{n+1}{n}\right)^{p-2},
\end{multline*}
where  $q^*=2p/(p+ 2)$.
\end{example}

\begin{example}
The mapping
\[
\varphi(z) = \sin z, \quad z = x + iy,
\]
is conformal and maps the rectangle $Q := (-\pi/2, \pi/2) \times (-d, d)$ onto the interior of the ellipse
\[
\Omega_e := \left\{w=u+iv: \left(\frac{u}{\cosh d}\right)^2 + \left(\frac{v}{\sinh d}\right)^2 = 1\right\}
\]
with slits from the foci $(\pm 1, 0)$ to the tips of the major semi-axes $(\pm \cosh d, 0)$.

Now we estimate the following quantities. A straightforward calculation
yields
\[
\|J_{\varphi}\mid L^{\infty}(Q)\|= \esssup\limits_{Q} \frac{1}{2} (\cos 2x + \cosh 2y) \leq \cosh^2 d,
\]
$$
|Q|=2 \pi d \quad \text{and \quad} |\Omega_e|=\frac{\pi \sinh 2d}{2}.
$$

Let us remind that for $1<q<2$ the constant $A_{p, q}(Q)$ can be estimated as
\[
A_{p, q}(Q)\leq \left( \dfrac{q - 1}{2 - q} \right)^{\!\frac{q - 1}{q}}
\dfrac{\left(\sqrt{\pi}\cdot \sqrt[q]{2}\right)^{-1}}
{\sqrt{\Gamma(2/q)\,\Gamma(3 - 2/q)}}.
\]

Then by Theorem \ref{th 3.6} we have
\begin{multline*}
\frac{1}{\lambda_p^{(1)}(\Omega_e)} \leq
 \inf_{q \in (q^*,2)} \left\{ A_{p, q}^p(Q) \cdot |Q|^{\frac{p(2 - q)}{2q}}\right\} \cdot |\Omega_e|^{\frac{p - 2}{2}} \cdot ||J_{\varphi}\,|\,L^{\infty}(Q)|| \\
\leq  \inf_{q \in (q^*,2)} \left\{ \left[\left( \dfrac{q - 1}{2 - q} \right)^{\!\frac{q - 1}{q}}
\dfrac{\left(\sqrt{\pi}\cdot \sqrt[q]{2}\right)^{-1}}
{\sqrt{\Gamma(2/q)\,\Gamma(3 - 2/q)}}\right]^p (2\pi d)^{\frac{p}{q}}\right\}
\left(\frac{\sinh 2d}{4d}\right)^{\frac{p}{2}} \frac{\coth d}{\pi},
\end{multline*}
where  $q^*=2p/(p+ 2)$.
\end{example}

In the case of conformal mappings $\varphi:\mathbb D\to\Omega$ Theorem~\ref{th 3.5} can be
reformulated as
\begin{theorem}\label{th-disc}
Let $\Omega$ be a conformal $\alpha$-regular domain about the unit disc $\mathbb D$.
Then for any $p>2$ the following inequality holds
\begin{equation*}
\frac{1}{\lambda_p^{(1)}(\Omega)} \leq \\
 \inf_{q \in (q^*,2]} \left\{ A_{\frac{p \alpha}{\alpha -1}, q}^p(\mathbb D) \cdot \pi^{\frac{p(2 - q)}{2q}}\right\} \cdot |\Omega|^{\frac{p - 2}{2}} \cdot ||J_{\varphi}\,|\,L^{\alpha}(\mathbb D)||,
\end{equation*}
where $q^*=2\alpha p/(\alpha p+ 2(\alpha-1))$ and $A_{\frac{p \alpha}{\alpha -1}, q}(\mathbb D)$ is as defined in~\eqref{SPC}.
\end{theorem}

\section{Spectral estimates in quasidiscs}
In this section, we refine Theorem~\ref{th-disc} for quasidiscs by using the integral estimates~\cite{GPU17_2} for the Jacobians of conformal mappings.

Recall that a homeomorphism $\varphi:\widetilde{\Omega}\to\Omega$ between planar domains is said to be $K$-quasiconformal
if $\varphi\in W^{1,2}_{\loc}(\widetilde{\Omega})$ and there exists a constant $1\leq K<\infty$ such that
$$
|D\varphi(z)|^2\leq K |J(z,\varphi)|\,\,\text{for almost all}\,\,z\in\widetilde{\Omega}.
$$

Recall that for any planar $K$-quasiconformal homeomorphism $\varphi:\widetilde{\Omega}\to\Omega$
the following sharp result is known: $J(z,\varphi)\in L^p_{\loc}(\widetilde{\Omega})$
for any $1 \leq p<\frac{K}{K-1}$ (see \cite{Ast,G81}). Hence for any conformal
mapping $\varphi:\mathbb D\to\Omega$ of the unit disc $\mathbb D$ onto a $K$-quasidisc $\Omega$ its Jacobian
$J(z,\varphi)$ belongs to $L^p(\mathbb D)$ for any $1 \leq p<\frac{K^2}{K^2-1}$ (see \cite{BGU15,GPU17_2}).

A domain $\Omega$ is called a $K$-quasidisc if it is the image of the unit disc $\mathbb D$
under $K$-quasiconformal homeomorphisms of the plane onto itself.

The following result was obtained in~\cite{GPU17_2}:
\begin{theorem}\label{Est_Der}
Let $\Omega\subset\mathbb C$ be a $K$-quasidisc, and let $\varphi:\mathbb D\to\Omega$ be a conformal mapping. Assume that  $1<\kappa<\frac{K^2}{K^2-1}$.
Then
\begin{equation*}\label{Ineq_2}
\left(\iint\limits_{\mathbb D} J(z,\varphi)^{\kappa}~dxdy \right)^{\frac{1}{\kappa}}
\leq \frac{C_\kappa^2 K^2 \pi^{\frac{1}{\kappa}-1}}{4}
\exp\left\{{\frac{K^2 \pi^2(2+ \pi^2)^2}{2\log3}}\right\}\cdot |\Omega|.
\end{equation*}
where
$$
C_\kappa=\frac{10^{6}}{[(2\kappa -1)(1- \nu)]^{1/2\kappa}}, \quad \nu = 10^{8 \kappa}\frac{2\kappa -2}{2\kappa -1}(24\pi^2K^2)^{2\kappa}<1.
$$
\end{theorem}

Combining Theorems~\ref{th-disc} and~\ref{Est_Der}, we obtain spectral estimates for the degenerate
$p$-Laplace operator with the Dirichlet boundary condition in quasidiscs.
\begin{theorem}\label{Th-quasidisc}
Let $\Omega\subset\mathbb C$ be a $K$-quasidisc. Then
$$
\lambda_p^{(1)}(\Omega)\geq \frac{M_p(K)}{|\Omega|^{\frac{p}{2}}}=\frac{M^{\ast}_p(K)}{R_{\ast}^{p}},
$$
where $R_{\ast}$ is a radius of a disc $\Omega^{\ast}$ of the same area as $\Omega$ and $M^{\ast}_p(K)=M_p(K)\pi^{-{p/2}}$.
\end{theorem}

The quantity $M_p(K)$ in Theorem A depends only on $p$ and a quasiconformity coefficient $K$ of $\Omega$:
\begin{multline}\label{M(K)}
M_p(K)= \frac{4\pi^{1+\frac{p}{2}}}{K^2} \exp\left\{-{\frac{K^2 \pi^2(2+ \pi^2)^2}{2\log3}}\right\} \times \\
{} \inf_{\alpha \in(1,\alpha*)} \inf_{q\in(q^*,2]}
\left\{A_{\frac{p \alpha}{\alpha -1},q}^{-p}(\mathbb D)\pi^{-\frac{p}{q}-\frac{1}{\alpha}} C_\alpha^{-2}\right\} , \\
{} C_\alpha=\frac{10^{6}}{[(2\alpha -1)(1- \nu(\alpha))]^{1/2\alpha}},
\end{multline}
where $q^*=2\alpha p/(\alpha p+ 2(\alpha-1))$, $\alpha*=\min \left({\frac{K^2}{K^2-1}, \widetilde{\alpha}}\right)$, and where $\widetilde{\alpha}$ is the unique solution of the equation $\nu(\alpha)=1$.

\begin{proof}
Given that, for $K \geq 1$, $K$-quasidiscs are conformal $\alpha$-regular domains for $1<\alpha<\frac{K^2}{K^2-1}$ \cite{GU16}. Then by Theorem~\ref{th-disc} for any $1<\alpha<\frac{K^2}{K^2-1}$ we have
\begin{equation}\label{Est_1}
\frac{1}{\lambda_p^{(1)}(\Omega)} \leq \\
 \inf_{q \in (q^*,2]} \left\{ A_{r, q}^p(\mathbb D) \cdot \pi^{\frac{p(2 - q)}{2q}}\right\} \cdot |\Omega|^{\frac{p - 2}{2}} \cdot ||J_{\varphi}\,|\,L^{\alpha}(\mathbb D)||,
\end{equation}
where $q^*=2\alpha p/(\alpha p+ 2(\alpha-1))$.

Now we estimate the integral from the right-hand side of this inequality. According to Theorem~\ref{Est_Der} we obtain
\begin{multline}\label{Est_2}
\|J_{\varphi}\,|\,L^\alpha(\mathbb D)\|=
\left(\iint\limits_{\mathbb D} J(z,\varphi)^\alpha~dxdy\right)^{\frac{1}{\alpha}} \\
{} \leq \frac{C_\alpha^2 K^2 \pi^{\frac{2}{\alpha}-1}}{4}
\exp\left\{{\frac{K^2 \pi^2(2+ \pi^2)^2}{2\log3}}\right\}\cdot |\Omega|.
\end{multline}

Combining inequalities~\eqref{Est_1} and \eqref{Est_2}
we get the required inequality.
\end{proof}

As an application of Theorem~\ref{Th-quasidisc} we obtain the lower estimates of the first
eigenvalue of the Dirichlet eigenvalue problem for the degenerate
$p$-Laplace operator in star-shaped and spiral-shaped domains.

A simply connected domain $\Omega^*$ is $\beta$-star-shaped (with respect to $z_0=0$)
if the function $\varphi(z)$, $\varphi(0)=0$, conformally maps a unit disc $\mathbb{D}$ onto $\Omega^*$ and satisfies the condition \cite{FKZ}:
\[
\left|\arg \frac{z \varphi^{\prime}(z)}{\varphi(z)} \right| \leq \beta \pi/2, \quad 0 \leq \beta <1, \quad |z|<1.
\]

A simply connected domain $\Omega_s$ is $\beta$-spiral-shaped (with respect to $z_0=0$)
if the function $\varphi(z)$, $\varphi(0)=0$, conformally maps a unit disc $\mathbb{D}$ onto $\Omega_s$ and satisfies the condition \cite{S87}:
\[
\left|\arg e^{i \delta} \frac{z \varphi^{\prime}(z)}{\varphi(z)} \right| \leq \beta \pi/2, \quad 0 \leq \beta <1, \quad |\delta|<\beta \pi/2, \quad |z|<1.
\]

In \cite{FKZ} and \cite{S87}, respectively, it is shown that
boundaries of domains $\Omega^*$ and $\Omega_s$ are a $K$-quasicircles with $K=\cot ^2(1-\beta)\pi/4$.

Setting $\Omega = \Omega^*$ or $\Omega = \Omega_s$ Theorem~\ref{Th-quasidisc} implies
\begin{multline*}
\frac{1}{\lambda_p^{(1)}(\Omega)} \leq
\inf_{\alpha \in \left(1,\frac{1}{1-\tan ^4(1-\beta) \frac{\pi}{4}}\right)} \inf_{q\in(q^*,2]}
\left\{A_{\frac{p \alpha}{\alpha -1},q}^{p}(\mathbb D)\pi^{\frac{p}{q}+\frac{1}{\alpha}} C_\alpha^{2}\right\} \\
\times
 \frac{\cot ^4(1-\beta)\frac{\pi}{4}}{4\pi^{1+\frac{p}{2}}}
\exp\left\{{\frac{\pi^2(2+ \pi^4)^2 \cot ^4(1-\beta)\frac{\pi}{4}}{2\log3}}\right\}\cdot \big|\Omega\big|^{\frac{p}{2}},
\end{multline*}
where $q^*=2\alpha p/(\alpha p+ 2(\alpha-1))$,
$$
C_\alpha=\frac{10^{6}}{[(2\alpha -1)(1- \nu)]^{1/\alpha}}, \quad \nu = 10^{8 \alpha} (24\pi^2 \cot ^4(1-\beta)\pi/4)^{\alpha}\frac{2\alpha -2}{2\alpha -1}<1.
$$

\end{document}